\documentclass[12pt,reqno]{amsart}
\usepackage{indentfirst, amssymb, amsmath, amsthm, mathrsfs, setspace, indentfirst, enumerate,  mathrsfs, amsmath, amsthm}
\usepackage[bookmarksnumbered, colorlinks, plainpages]{hyperref}
\usepackage{mathrsfs}
\usepackage{cite}

\newtheorem*{theoA}{Theorem A}
\newtheorem*{theoB}{Theorem B}
\newtheorem*{theoC}{Theorem C}
\newtheorem*{theoD}{Theorem D}
\newtheorem*{theoE}{Theorem E}
\newtheorem*{theoF}{Theorem F}
\newtheorem*{theoG}{Theorem G}
\newtheorem*{theoH}{Theorem H}

\newtheorem*{cor A}{Corollary A}
\newtheorem*{cor B}{Corollary B}

\newtheorem{theo}{Theorem}[section]
\newtheorem{lem}{Lemma}[section]

\newtheorem{prob}{Problem}[section]

\newtheorem{rem}{Remark}[section]

\newcommand{\ol}{\overline}
\newcommand{\be}{\begin{equation}}
\newcommand{\ee}{\end{equation}}
\newcommand{\beas}{\begin{eqnarray*}}
\newcommand{\eeas}{\end{eqnarray*}}
\newcommand{\bea}{\begin{eqnarray}}
\newcommand{\eea}{\end{eqnarray}}

\numberwithin{equation}{section}
\begin{document}
\title[S\MakeLowercase{harp} B\MakeLowercase{ohr} \MakeLowercase{Radii for} S\MakeLowercase{chwarz Functions and.... }]{\LARGE S\MakeLowercase{harp} B\MakeLowercase{ohr} \MakeLowercase{Radii for} S\MakeLowercase{chwarz Functions and } \MakeLowercase{Directional derivative Operators in} $\mathbb{C}^n$}
\date{}
\author[M. B. A\MakeLowercase{hamed}, S. M\MakeLowercase{ajumder} \MakeLowercase{and} D. P\MakeLowercase{ramanik}]{M\MakeLowercase{olla} B\MakeLowercase{asir} A\MakeLowercase{hamed}$^*$, S\MakeLowercase{ujoy} M\MakeLowercase{ajumder} \MakeLowercase{and} D\MakeLowercase{ebabrata} P\MakeLowercase{ramanik}}

\address{Molla Basir Ahamed,
	Department of Mathematics,
	Jadavpur University,
	Kolkata-700032, West Bengal, India.}
\email{mbahamed.math@jadavpuruniversity.in}

\address{Sujoy Majumder, Department of Mathematics, Raiganj University, Raiganj, West Bengal-733134, India.}
\email{sm05math@gmail.com}

\address{Debabrata Pramanik, Department of Mathematics, Raiganj University, Raiganj, West Bengal-733134, India.}
\email{debumath07@gmail.com}

\renewcommand{\thefootnote}{}
\footnote{*\emph{Corresponding Author}: Molla Basir Ahamed.}

\renewcommand{\thefootnote}{\arabic{footnote}}
\setcounter{footnote}{0}

\begin{abstract}
This paper is devoted to the investigation of multidimensional analogues of refined Bohr-type inequalities for bounded holomorphic mappings on the unit polydisc $\mathbb{P}\Delta(0;1_n)$. We provide a definitive resolution to the Bohr phenomenon in several complex variables by determining sharp radii for functional power series involving the class of Schwarz functions $\omega_{n,m}\in\mathcal{B}_{n,m}$ and the local modulus $|f(z)|$. By employing the directional derivative operator $\partial_uf(z) = \sum_{k=1}^{n} u_k \frac{\partial f(z)}{\partial z_k}$, where $u=(u_1,u_2,\ldots,u_n)\in\mathbb{C}^n$ such that $|u_1|+|u_2|+\ldots+|u_n|=1$, we obtain refined growth estimates for derivatives that generalize well-known univariate results to $\mathbb{C}^n$. The optimality of the obtained constants is rigorously verified, demonstrating that all established radii are sharp.
\end{abstract}
\thanks{Typeset by \AmS -\LaTeX}
\maketitle

\section{\bf Introduction and Preliminaries}
The classical theorem of Harald Bohr \cite{Bohr-PLMS-1914}, first established over a century ago, continues to inspire extensive research on what is now known as the Bohr phenomenon. Renewed interest in this topic emerged in the $1990$s following successful extensions to holomorphic functions of several complex variables and more abstract functional-analytic settings. In particular, in $1997$, Boas and Khavinson \cite{Boas-Khavinson-PAMS-1997} introduced and determined the $n$-dimensional Bohr radius for the family of holomorphic functions bounded by unity on the polydisc (see Section \ref{Sub-Sec-1.3} for a detailed discussion). This seminal contribution stimulated significant interest in Bohr-type problems across a wide range of mathematical disciplines. Subsequent studies have produced further advances regarding the Bohr phenomenon for multidimensional power series, with notable contributions by Aizenberg \cite{Aizen-PAMS-2000, Aizenberg, Aizenberg-SM-2007}, Aizenberg \textit{et al.} \cite{Aizenberg-PAMC-1999, Aizenberg-SM-2005, Aizenberg-Aytuna-Djakov-JMAA-2001}, Defant and Frerick \cite{Defant-Frerick-IJM-2006}, and Djakov and Ramanujan \cite{Djakov-Ramanujan-JA-2000}. A comprehensive overview of the different aspects and generalizations of Bohr's inequality can be found in \cite{Lata-Singh-PAMS-2022, Liu-Pon-PAMS-2021, Ali-Abu-Muhanna-Ponnusamy, Alkhaleefah-Kayumov-Ponnusamy-PAMS-2019, Defant-Frerick-AM-2011, Hamada-IJM-2012, Paulsen-Popescu-Singh-PLMS-2002, Paulsen-Singh-PAMS-2004, Paulsen-Singh-BLMS-2006,
	Ahamed-Ahammed-CMFT-2023,Ahamed-Ahammed-MJM-2024,Ahamed-Allu-Halder-BSM-2025,Ahamed-Allu-Halder-CRM-2025,Ahammed-Ahamed-CMB-2024,Ahammed-Ahamed-CVEE-2024,Ahammed-Ahamed-MJM-2025,Ahammed-Ahamed-Roy-Filomat-2025}, as well as in the monograph by Kresin and Maz'ya \cite{Kresin-1903} and the references therein. In particular, \cite[Section 6]{Kresin-1903}, which is devoted to Bohr-type theorems, highlights promising directions for extending several classical inequalities to holomorphic functions of several complex variables and, more significantly, to solutions of partial differential equations.\vspace{1.2mm}

The transition from the unit disk $\mathbb{D}$ to the polydisc $\mathbb{P}\Delta(0;1_n)$ is not merely a matter of indexing; it involves a fundamental shift in the underlying complex geometry. While the classical Bohr theorem of 1914 provides an elegant bound for power series on the unit disk $\mathbb{D}$, transitioning to the polydisc $\mathbb{P}\Delta(0;1_{n})$ involves a fundamental shift in complex geometry that presents significant Dimensional Shift Limitations. A critical challenge within the realm of Several Complex Variables (SCV) is determining whether univariate refinements?such as those involving Rogosinski radii or area-based estimates?retain their sharpness when subjected to multidimensional operators like the directional derivative. Furthermore, although recent studies have established Bohr-type inequalities for holomorphic functions in the unit disk using Schwarz functions, it has remained an open problem whether these results could be extended to functions defined on the unit polydisc. This paper addresses these gaps by determining sharp multidimensional Bohr radii for functional power series involving the class of Schwarz functions $\omega_{n,m}\in\mathcal{B}_{n,m}$ and the directional derivative operator, providing a definitive resolution to the Bohr phenomenon in the SCV setting. 
\subsection{\bf Classical Bohr inequality and its recent implications}
Let $ \mathcal{B} $ denote the class of analytic functions in the unit disk $\mathbb{D}:=\{{\zeta}\in\mathbb{C} : |{ \zeta}|<1\} $ of the form $ f({ \zeta})=\sum_{k=0}^{\infty}a_k {\zeta}^k $ such that $ |f({ \zeta})|<1 $ in $ \mathbb{D} $.  In the study of Dirichlet series, in $ 1914$, Harald Bohr \cite{Bohr-PLMS-1914} discovered the following interesting phenomenon 
\begin{theoA} If $ f\in\mathcal{B}$, then the following sharp inequality holds:
\begin{align}
\label{b1}	M_f(r):=\sum_{k=0}^{\infty}|a_k|r^k\leq 1\;\;\mbox{for}\;\; |{ \zeta}|=r\leq \frac{1}{3}.
\end{align}
\end{theoA}

Inequality (\ref{b1}), along with the sharp constant $1/3$, is known as the \textit{classical Bohr inequality}, and $1/3$ is referred to as the \textit{Bohr radius} for the family $\mathcal{B}$. Bohr \cite{Bohr-PLMS-1914} originally established this inequality for $r \leq 1/6$, and Wiener subsequently proved that the constant $1/3$ is sharp. Since then, several alternative proofs have appeared (see \cite{Paulsen-Popescu-Singh-PLMS-2002, Paulsen-Singh-BLMS-2006, Sidon-MZ-1927, Tomic-MS-1962}, as well as the survey \cite{Ali-Abu-Muhanna-Ponnusamy} and \cite[Chapter 8]{Garcia-Mashreghi-Ross}). Many of these works employ methods from complex analysis, functional analysis, number theory, and probability, further developing the theory and applications of Bohr's ideas to Dirichlet series. For instance, various multidimensional generalizations of this result have been obtained in \cite{Aizenberg-PAMC-1999, Aizenberg-Aytuna-Djakov-JMAA-2001, Aizenberg, Boas-Khavinson-PAMS-1997, Djakov-Ramanujan-JA-2000, Jia-Liu-Ponnusamy-AMP-2025}.
\vspace{1.2mm}

It is noteworthy that if $|a_0|$ in the Bohr inequality is replaced by $|a_0|^2$, the Bohr radius increases from $1/3$ to $1/2$. Moreover, if $a_0=0$ in Theorem A, the sharp Bohr radius is further improved to $1/\sqrt{2}$ (see, e.g., \cite{Kayumov-Ponnusamy-CMFT-2017}, \cite[Corollary 2.9]{Paulsen-Popescu-Singh-PLMS-2002}, and the recent work \cite{Ponnusamy-Wirths-CMFT-2020} for a more general result). These refinements rely on sharp coefficient estimates of the form\begin{align*}|a_n| \leq 1 - |a_0|^2, \quad n \geq 1, ; f \in \mathcal{B}.\end{align*}Applying these inequalities, Kayumov and Ponnusamy \cite{Kayumov-Ponnusamy-CMFT-2017} observed that the sharp form of Theorem A cannot be obtained in the extremal case $|a_0| < 1$. Nevertheless, a sharp version of Theorem A has been established for each individual function in $\mathcal{B}$ (see \cite{Alkhaleefah-Kayumov-Ponnusamy-PAMS-2019}), as well as for several subclasses of univalent functions (see \cite{Aizenberg, Muhanna-CVEE-2010}).\vspace{1.2mm}

The Bohr--Rogosinski radius serves as a natural analogue to the classical Bohr radius. This notion was originally established by Rogosinski \cite{Rogosinski-1923} within the context of the Schur class $\mathcal{B}$. For a function $f \in \mathcal{B}$, we consider the Bohr--Rogosinski sum $R^f_N(z)$ defined by\begin{align}\label{e-11.12}R^f_N(z) := |f(z)| + \sum_{n=N}^{\infty} |a_n| r^n, \quad |z| = r.\end{align}Note that for the case $N=1$, the expression in \eqref{e-11.12} generalizes the classical Bohr sum by replacing the fixed initial coefficient $|a_0| = |f(0)|$ with the local modulus $|f(z)|$. The functional inequality $R^f_N(z) \leq 1$ constitutes the Bohr--Rogosinski inequality.

Recently, the classical Bohr's inequality is improved in following form.
\begin{theoB}\cite{Kayumov-Ponnusamy-ArXiv-2018} Suppose that $f(z)=\sum_{k=0}^{\infty} a_k z^k$  is analytic in $\mathbb{D}$ and $|f(z)| \le 1$ in $\mathbb{D}$. Then
\begin{enumerate}
	\item[\emph{(a)}] $|f(z)| + \sum_{k=1}^{\infty} |a_k| r^k \le 1 \quad \text{for } |z| = r \le \sqrt{5} - 2 \approx 0.236068.$ The constant $\sqrt{5} - 2$ is the best possible. Moreover,
	\item[\emph{(b)}] $|f(z)|^{2} + \sum_{k=1}^{\infty} |a_k| r^k \le 1 \quad \text{for } |z| = r \le \frac{1}{3}.$ The constant $\frac{1}{3}$ is the best possible.
\end{enumerate}
\end{theoB}

Motivated by the work of Kayumov and Ponnusamy \cite{Kayumov-Ponnusamy-CRMASP-2018}, Bohr-type inequalities for the family $\mathcal{B}$ were investigated in \cite{Liu-Shang-Xu-JIA-2018}. This investigation considered several forms where the Taylor coefficients in the classical Bohr inequality are partially or entirely replaced by higher-order derivatives of $f$. We recall only one of these results here.
\begin{theoC} \cite[Theorem 2.1]{Liu-Shang-Xu-JIA-2018} Suppose that $f\in \mathcal{B}$ and $f({ \zeta})=\sum_{n=0}^{\infty} a_n \zeta^n.$ Then the following sharp inequality holds:
\begin{align*}
|f({ \zeta})|+|f'({ \zeta})|r+\sum_{n=2}^{\infty} |a_n|r^n\leq 1\quad \text{for}\;|{ \zeta}|=r\leq \frac{\sqrt{17}-3}{4}.
\end{align*}
\end{theoC}

In $2022$, Wu \textit{et al.} \cite{Wu-Wang-Long-2022} obtained a Bohr-type inequality, formulated in terms of convex combinations of the coefficients of the power series terms. Motivated by this approach, we recall Theorem D below, which generalizes the classical Bohr inequality via a convex combination framework.

\begin{theoD}\cite[Theorem 3.1]{Wu-Wang-Long-2022}
Suppose that $ f(z)=\sum_{n=0}^{\infty}a_nz^n $ is analytic in the unit disk $ \mathbb{D} $ and $ |f(z)|<1 $ in $ \mathbb{D} $.
Then for arbitrary $t\in[0,1]$, it holds that
\begin{align*}
t|f(z)|+(1-t)\sum_{k=0}^{\infty}|a_{k}|\, r^{k} \le 1 \quad \text{for } r\le R_1,
\end{align*}
where the radius 
\begin{align*}
R_1=
\begin{cases}
\dfrac{1-2\sqrt{1-t}}{4t-3}, & t\in\left[0,\dfrac34\right)\cup\left(\dfrac34,1\right],\\[2ex]
\dfrac12, & t=\dfrac34.
\end{cases}
\end{align*}
is the best possible.
\end{theoD}

In the same work, Wu \textit{et al.} \cite{Wu-Wang-Long-2022} derived a result for computing the Bohr-type radius of analytic functions $f(z)$, where the coefficients $|a_0|$ and $|a_1|$ are replaced by the function value $|f(z)|$ and its derivative $|f'(z)|$, respectively.

\begin{theoE} \cite[Theorem 3.4]{Wu-Wang-Long-2022}
Suppose that $ f(z)=\sum_{n=0}^{\infty}a_nz^n $ is analytic in the unit disk $ \mathbb{D} $ and $ |f(z)|<1 $ in $ \mathbb{D} $.
Then for any $\lambda\in(0,+\infty)$, the inequality
\begin{align*}
|f(z)|+|f'(z)|r+\lambda\sum_{k=2}^\infty |a_k|r^k \le 1 \quad {for}\; r\le R_3,
\end{align*}
where
\begin{align*}
R_3=R_3(\lambda)=
\begin{cases}
r_\lambda, & \lambda\in\left(\dfrac12,+\infty\right),\vspace{1.2mm}\\[6pt]
r_*, & \lambda\in\left(0, \dfrac12\right],
\end{cases}
\end{align*}
and $r_\lambda$ and $r_*\approx 0.3191$ are the unique positive real roots in the interval
$(0,\sqrt{2}-1)$ of the equations
\begin{align*}
2\lambda r^4+(4\lambda-1)r^3+(2\lambda-1)r^2+3r-1=0
\end{align*}
and $r^4+r^3+3r-1=0,$ respectively. Moreover, the radius $R_3$ is best possible.
\end{theoE}

Let 
\begin{align*}
B_m = \Bigl\{ \omega \in \mathcal{B} : \omega(0) = \omega'(0) = \cdots = \omega^{(m-1)}(0) = 0, \ \omega^{(m)}(0) \neq 0 \Bigr\}
\end{align*}
be the class of Schwarz functions, where $m\in \mathbb{N}$.\vspace{1.2mm}

Utilizing the class of Schwarz functions $\omega\in B_m$, Hu et al. \cite{Hu-Wang-Long-2022} further extend Theorem D and E as follows.

\begin{theoF}\cite[Theorem 3.1]{Hu-Wang-Long-2022} Suppose that $f \in \mathcal{B}$ with $f(z) = \sum_{k=0}^{\infty} a_k z^k$, $a := |a_0|$,
and let $\omega \in B_m$ for some $m \in \mathbb{N}$. Then for $t \in [0,1)$, we have
\begin{align*}
t\, |f(\omega(z))| + (1-t) \sum_{k=0}^{\infty} |a_k|\, |\omega(z)|^k \le 1
\end{align*}
for $|z| = r \le R_{t,m}$, where
\begin{align*}
R_{t,m} =
\begin{cases}
\sqrt[m]{\dfrac{1-2\sqrt{1-t}}{4t-3}}, & t \in \left[0, \dfrac{3}{4}\right)\cup\left(\dfrac{3}{4},1\right), \\[2mm]
\sqrt[m]{\dfrac{1}{2}}, & t = \frac{3}{4}.
\end{cases}
\end{align*}
The radius $R_{t,m}$ is best possible.
\end{theoF}

\begin{theoG}\cite[Theorem 3.4]{Hu-Wang-Long-2022} Suppose that $f \in \mathcal{B}$ with $f(z) = \sum_{k=0}^{\infty} a_k z^k$, $a := |a_0|$, and let $\omega \in B_m$ for some $m \in \mathbb{N}$. Then for $\lambda \in (0,\infty)$, we have
\begin{align*}
|f(\omega(z))| + |f'(\omega(z))|\,|\omega(z)| + \lambda \sum_{k=2}^{\infty} |a_k|\, |\omega(z)|^k \le 1
\end{align*}
for $|z| = r \le R_\lambda$, where
\begin{align*}
R_\lambda =
\begin{cases}
r_\lambda, & \lambda \in \left(\frac{1}{2}, \infty\right), \\[1mm]
r^*, & \lambda \in \left(0, \frac{1}{2}\right],
\end{cases}
\end{align*}
and the radius $R_\lambda$ is best possible. The radii $r_\lambda$ and $r^*$ are the unique positive roots in $\left(0, \sqrt[m]{\sqrt{2}-1}\right)$ of the equations
\begin{align*}
2 \lambda r^{4m} + (4\lambda -1) r^{3m} + (2\lambda -1) r^{2m} + 3 r^m -1 = 0
\end{align*}
and
\begin{align*}
r^{4m} + r^{3m} + 3 r^m -1 = 0,
\end{align*}
respectively.
\end{theoG}

\begin{theoH}\cite[Theorem 3.5]{Hu-Wang-Long-2022} Suppose that $f \in \mathcal{B}$ with $f(z) = \sum_{k=0}^{\infty} a_k z^k$, $a := |a_0|$,
and let $\omega \in B_m$ for some $m \in \mathbb{N}$. Then for $\lambda \in (0,\infty)$, we have
\begin{align*}
|f(\omega(z))|^2 + |f'(\omega(z))|\,|\omega(z)| + \lambda \sum_{k=2}^{\infty} |a_k|\, |\omega(z)|^k \le 1
\end{align*}
for $|z| = r \le R_{2,\lambda}$, where
\begin{align*}
R_{2,\lambda} =
\begin{cases}
r_{2,\lambda}, & \lambda \in (1, \infty), \\[1mm]
r_2^*, & \lambda \in (0, 1],
\end{cases}
\end{align*}
and the radius $R_{2,\lambda}$ is best possible. The radii $r_{2,\lambda}$ and $r_2^*$ are the unique positive roots in 
$\left(0, \sqrt[m]{\dfrac{\sqrt{5}-1}{2}}\right)$ of the equations
\begin{align*}
\lambda r^{4m} + (2\lambda -1) r^{3m} + \lambda r^{2m} + 2 r^m -1 = 0
\end{align*}
and
\begin{align*}
r^{4m} + r^{3m} + r^{2m} + 2 r^m - 1 = 0,
\end{align*}
respectively.
\end{theoH}

For recent developments regarding the Bohr radius for various functional classes, we refer the reader to \cite{Ahamed-RM-2023, Ahamed-Allu-BMMSS-2022, Ahamed-Allu-CMB-2023, Allu-Arora-JMAA-2023, Das-JMAA-2022, Gangania-Kumar-MJM-2022, Hamada-AAMP-2025} and the extensive references therein. Given the significance of these results in the unit disk, it is natural to investigate their higher-dimensional analogues. This leads us to the following problem:
\begin{prob}\label{Qn-4.1}Is it possible to establish Theorems F--H for holomorphic functions defined on the unit polydisc $\mathbb{P}\Delta(0; \mathbf{1}_n)$ in terms of the associated Schwarz functions?
\end{prob}

The main objective of this work is to develop multidimensional counterparts of Theorems F-H. Specifically, Theorem \ref{Th-2.3} extends Theorem F to the setting of several complex variables through the use of Schwarz functions $\omega \in \mathcal{B}_{n,m}$. In addition, Theorems \ref{Th-2.4} and \ref{Th-2.5} generalize Theorems G and H, respectively, by replacing the classical univariate derivative with the directional derivative operator $\partial_uf(z)$, which constitutes the natural geometric extension in higher dimensions.

\subsection{\bf Basic Notations in several complex variables}\label{Sub-Sec-1.3}
For $z=(z_1,\ldots,z_n)$ and $w=(w_1,\ldots,w_n)$ in $\mathbb{C}^{n}$, we denote $\langle z,w\rangle=z_1\ol w_1+\ldots+z_n \ol w_n$ and $||z||=\sqrt{\langle z,z\rangle}$. The absolute value of a complex number $z_1$ is denoted by $|z_1|$ and for $z\in\mathbb{C}^n$, we define
\begin{align*}
	||z||_{\infty}=\max\limits_{1\leq i\leq n}|z_i|.
\end{align*}

Throughout the paper, for $z=(z_1,\ldots,z_n)\in\mathbb{C}^n$ and $p\in\mathbb{N}$, we define $z^p=(z_1^p,\ldots,z_n^p)$.
An open polydisk (or open polycylinder) in $\mathbb{C}^n$ is a subset $\mathbb{P}\Delta(a;r)\subset \mathbb{C}^n$ of the form 
\[\mathbb{P}\Delta(a;r)=\prod\limits_{j=1}^n \Delta(a_j;r_j)=\lbrace z\in\mathbb{C}^n: |z_i-a_i|<r_i,\;i=1,2,\ldots,n\rbrace,\]
the point $a=(a_1,\ldots,a_n)\in\mathbb{C}^n$ is called the centre of the polydisk and $r=(r_1,\ldots,r_n)\in\mathbb{R}^n\;(r_i>0)$ is called the polyradius. It is easy to see that
\begin{align*}
	\mathbb{P}\Delta(0;1)=\mathbb{P}\Delta(0;1_n)=\prod\limits_{j=1}^n \Delta(0_n;1_n).
\end{align*}

The closure of $\mathbb{P}\Delta(a;r)$ will be called the closed polydisk with centre $a$ and polyradius $r$ and will be denoted by $\ol{\mathbb{P}\Delta}(a;r)$. 

\smallskip
A multi-index $\alpha=(\alpha_1,\ldots,\alpha_n)$ of dimension $n$ consists of n non-negative integers $\alpha_j,\;1\leq j\leq n$; the degree of a multi-index $\alpha$ is the sum $|\alpha|=\sum_{j=1}^n \alpha_j$ and we denote $\alpha!=\alpha_1!\ldots \alpha_n!$. For $z=(z_1,\ldots,z_n)\in\mathbb{C}^n$ and a multi-index $\alpha=(\alpha_1,\ldots,\alpha_n)$, we define 
\[z^{\alpha}=\prod\limits_{j=1}^n z_j^{\alpha_j}\;\;\text{and}\;\;|z|^{\alpha}=\prod\limits_{j=1}^n |z_j|^{\alpha_j}.\]

For two multi-indexes $\alpha=(\alpha_1,\ldots,\alpha_n)$ and $\nu=(\nu_1,\ldots,\nu_n)$, we define $\nu^{\alpha}=\nu_1^{\alpha_1}\ldots \nu_n^{\alpha_n}.$ 
Let $f$ be a holomorphic function in a domain $\Omega\subset \mathbb{C}^n$, and $c=(c_1,\ldots,c_n)\in \Omega$. Then in a polydisk $\mathbb{P}\Delta(c;r)\subset \Omega$ with centre $c$, $f(z)$ has a power series expansion in $z_1-c_1,\ldots,z_n-c_n$,
\begin{align*}
	f(z)&=\sum\limits_{\alpha_1,\alpha_2,\ldots,\alpha_n=0}^{\infty} a_{\alpha_1,\alpha_2,\ldots,\alpha_n}(z_1-c_1)^{\alpha_1}(z_2-c_2)^{\alpha_2}\ldots (z_n-c_n)^{\alpha_n}\\&=
	\sum\limits_{|\alpha|=0} a_{\alpha}(z-c)^{\alpha}=\sum\limits_{|\alpha|=0}^{\infty} P_{|\alpha|}(z-c),
\end{align*}
which is absolutely convergent in $\mathbb{P}\Delta(c;r)$, where the term $P_k(z-c)$ is a homogeneous polynomial of degree $k$. \vspace{1.2mm}

One of our aims is to establish multidimensional analogues of Theorems F-G. For this purpose, we consider $f$ to be a holomorphic function in $\Omega \subset \mathbb{C}^n$ and define
\[ S_n=\{(u_1,u_2,\ldots,u_n)\in \mathbb{C}^{n}: |u_1|+|u_2|+\cdots+|u_n|=1\}.\]

For a differentiable function $f$ and a direction $u=(u_1,u_2,\ldots,u_n) \in S_n$ the directional derivative of $f$ along $u$ is defined by $\partial_{u}(f)$ of $f$ along a direction  is defined by
\[\partial_u(f)=\sideset{}{_{j=1}^{n}}{\sum} u_{j}\partial_{z_j}(f).\]
\par The $k$-th order directional derivative  $\partial^k_u(f)$ of $f$ along $u$ is then defined inductively by
$$\partial^k_u(f)=\partial_u(\partial_u^{k-1}(f)),\;\;\; k\in \mathbb{Z}^+$$
with the base case $\partial^1_u(f)=\partial_u(f).$\vspace{1.2mm}

Let $G\not=\varnothing$ be an open subset of $\mathbb{C}^n$. Let $f$ be a holomorphic function on $G$. For a point $a\in\mathbb{C}^n$, we write $f(z)=\sum_{i=0}^{\infty}P_i(z-a)$, where the term $P_i(z-a)$ is either identically zero or a homogeneous polynomial of degree $i$. Denote the zero-multiplicity of $f$ at $a$ by $k=\min\{i:P_i(z-a)\not\equiv 0\}$. Clearly $1$ is the zero-multiplicity of $f$ at $a$ when $f(a)=0$ and $\frac{\partial f(a)}{\partial z_j}\neq 0$ for some $j=1,2,\ldots,n$.\vspace{1.2mm}

Let $\Delta_{z_k}(0;1)$ be the unit disk on the $z_k$-plane. Let $\omega_i:\Delta_{z_k}(0;1)\to \mathbb{C}$ such that $|\omega_i(z)|\leq 1$ for all $z\in \Delta_{z_k}(0;1)$ and $\omega_i(0)=\omega_i^{(1)}(0)=\ldots=\omega_i^{(m-1)}(0)=0$ and $\omega_i^{(m)}(0)\neq 0$, where $i=1,2,\ldots, n$.
In what follows, let $z = (z_1, \dots, z_n)$ denote an element of $\mathbb{C}^n$ and let
\begin{align*}
	{ \mathcal{B}_{n,m}=\{\omega(z)\in\mathbb{C}^n:\omega(z)=(\omega_1(z_1),\omega_2(z_2),\ldots,\omega_n(z_n))\}.}
\end{align*}

\section{{\bf Main Results}}\label{Sec-2}
Theorem \ref{Th-2.3} below establishes a weighted Bohr-type inequality for holomorphic functions in the polydisk composed with mappings from $\mathcal{B}_{n,m}$. The parameter $t\in[0,1]$ interpolates between the pointwise estimate and the full Bohr sum. We further determine the sharp radius $R_{m,n,t}$, thereby extending the classical Bohr radius problem to a more general multidimensional setting.

\begin{theo}\label{Th-2.3} Let $f(z)=\sum_{|\alpha|=0} a_{\alpha}z^{\alpha}$ be a holomorphic function in the polydisk $\mathbb{P}\Delta(0;1_n)$ such that $|f(z)|\leq 1$ for all $z\in \mathbb{P}\Delta(0;1/n)$. Suppose $z=(z_1,\ldots,z_n)\in \mathbb{P}\Delta(0;1/n)$ and $r=(r_1,r_2,\ldots,r_n)$ such that $||z||_{\infty}={\bf r}$. Let $\omega\in \mathcal{B}_{n,m}$ for some $m\in\mathbb{N}$. Then for arbitrary $t\in [0,1]$, it holds that
\begin{align*}
\mathcal{A}(z, {\bf r}):=t|f(\omega(z))|+(1-t)\sum_{k=0}^{\infty}\sum\limits_{|\alpha|=k}|a_{\alpha}|\, |\omega(z)|^{\alpha}\leq 1\nonumber
\end{align*}
for ${\bf r}\leq R_{m,n,t}$, where 
\begin{align*}
R_{m,n,t}=
\begin{cases}
\sqrt[m]{\dfrac{1-2\sqrt{1-t}}{n(4t-3)}}, &t \in \left[0,\dfrac{3}{4}\right)\cup \left(\dfrac{3}{4},1\right]\vspace{1.2mm}\\[6pt]
\sqrt[m]{\dfrac{1}{2n}}, &t=\frac{3}{4}.
\end{cases}
\end{align*}
The radius $R_{m,n,t}$ is best possible. 
\end{theo}

\begin{rem} In the case where $\omega(z)=z$ for $z\in\mathbb{C}^n$ and $t=0$, Theorem \ref{Th-2.3} recovers the classical Bohr radius problem in several complex variables.
\end{rem}

While extending derivative-based inequalities to $\mathbb{C}^n$ often introduces significant geometric complexities, the directional derivative operator $\partial_u f(z)$ serves as the natural multidimensional counterpart to the univariate complex derivative. 
By utilizing this operator, we establish sharp growth estimates that extend known univariate bounds to the polydisc $\mathbb{P}\Delta(0;1_n)$. Consequently, we present Theorem \ref{Th-2.4} as the formal multidimensional generalization of Theorem G.

\begin{theo}\label{Th-2.4} Let $f(z)=\sum_{|\alpha|=0} a_{\alpha}z^{\alpha}$ be a holomorphic function in the polydisk $\mathbb{P}\Delta(0;1_n)$ such that $|f(z)|\leq 1$ for all $z\in \mathbb{P}\Delta(0;1/n)$. Suppose $z=(z_1,\ldots,z_n)\in \mathbb{P}\Delta(0;1/n)$ and $r=(r_1,r_2,\ldots,r_n)$ such that $||z||_{\infty}={\bf r}$. Let $\omega\in \mathcal{B}_{n,m}$ for some $m\in\mathbb{N}$. Then for arbitrary $\lambda\in (0,+\infty)$, it holds that
\begin{align*}
\mathcal{B}(z, {\bf r}):=|f(\omega(z))|+|\partial_u(f(\omega(z)))|\;||\omega(z)||_{\infty}+\lambda \sum_{k=2}^{\infty}\sum\limits_{|\alpha|=k}|a_{\alpha}|\, |\omega(z)|^{\alpha}\leq 1
\end{align*}
for ${\bf r}\leq R_{m,n,\lambda}$, where  
\begin{align*}
R_{m,n,\lambda}=
\begin{cases}
{\bf r}_{\lambda}, & \lambda \in \left(\dfrac{1}{2}, +\infty\right) \\
{\bf r}_*, & \lambda \in \left(0, \dfrac{1}{2}\right]
\end{cases}
\end{align*}
is the best possible, and the radii ${\bf r}_{\lambda}$ and ${\bf r}_*$ are the unique positive real roots of the equations
\begin{align*}
2\lambda (n{\bf r}^m)^4 + (4\lambda - 1) (n{\bf r}^m)^3 + (2\lambda - 1) (n{\bf r}^m)^2 + 3(n{\bf r}^m) - 1 = 0
\end{align*}
and
\begin{align*}
(n{\bf r}^m)^4+(n{\bf r}^m)^3+ 3(n{\bf r}^m) - 1 = 0
\end{align*}
in the interval $\left(0, \sqrt[m]{\dfrac{\sqrt{2} - 1}{n}}\right)$ respectively.
\end{theo}

\begin{rem} Setting $\lambda=1$ in Theorem \ref{Th-2.4} recovers Corollary 4.5 \cite{Hu-Wang-Long-2021} in the one-dimensional setting, thereby extending that result to holomorphic functions in higher dimensions.
\end{rem}
\begin{rem} When $\omega(z)=z$, where $z\in\mathbb{C}^n$, Theorem \ref{Th-2.4} can be viewed as a higher-dimensional extension of Theorem E.
\end{rem}

\begin{rem} when $\omega(z)=z$, where $z\in\mathbb{C}^n$ and $\lambda=1$, Theorem \ref{Th-2.4} naturally extends Theorem C to the setting of several complex variables.
\end{rem}

\begin{theo}\label{Th-2.5} Let $f(z)=\sum_{|\alpha|=0} a_{\alpha}z^{\alpha}$ be a holomorphic function in the polydisk $\mathbb{P}\Delta(0;1_n)$ such that $|f(z)|\leq 1$ for all $z\in \mathbb{P}\Delta(0;1/n)$. Suppose $z=(z_1,\ldots,z_n)\in \mathbb{P}\Delta(0;1/n)$ and $r=(r_1,r_2,\ldots,r_n)$ such that $||z||_{\infty}={\bf r}$. Let $\omega\in \mathcal{B}_{n,m}$ for some $m\in\mathbb{N}$. Then for arbitrary $\lambda\in (0,+\infty)$, it holds that
\begin{align*}
\mathcal{C}(z, {\bf r}):=|f(\omega(z))|^2+|\partial_u(f(\omega(z)))|\;||\omega(z)||_{\infty}+\lambda \sum_{k=2}^{\infty}\sum\limits_{|\alpha|=k}|a_{\alpha}|\, |\omega(z)|^{\alpha}\leq 1
\end{align*}
for ${\bf r}\leq \tilde R_{m,n,\lambda}$, where  
\begin{align*}
\tilde R_{m,n,\lambda}=
\begin{cases}
\tilde {\bf r}_{\lambda}, & \lambda \in \left(1, +\infty\right) \\
\tilde {\bf r}_*, & \lambda \in \left(0,1\right]
\end{cases}
\end{align*}
is the best possible, and the radii $\tilde {\bf r}_{\lambda}$ and $\tilde {\bf r}_*$ are the unique positive real roots of the equations
\begin{align*}
\lambda (n{\bf r}^m)^4 + (2\lambda - 1) (n{\bf r}^m)^3 +\lambda(n{\bf r}^m)^2 + 2(n{\bf r}^m) - 1 = 0
\end{align*}
and
\begin{align*}
(n{\bf r}^m)^4+(n{\bf r}^m)^3++(n{\bf r}^m)^2+2(n{\bf r}^m) - 1 = 0
\end{align*}
in the interval $\left(0, \sqrt[m]{\dfrac{\sqrt{5} - 1}{2n}}\right)$ respectively.
\end{theo}

\section{{\bf Key lemmas}}\label{Sec-3}
In order to establish our main results, we need the following lemmas. The first of these is a special case of \cite[Theorem 2.2 ]{Chen-Hamada-Ponnusamy-Vijayakumar-JAM-2024}.
\begin{lem}\label{Lem1}Let $f$ be holomorphic in the polydisk $\mathbb{P}\Delta(0;1_n)$ such that $|f(z)|\le 1$ for all $z\in \mathbb{P}\Delta(0;1_n)$. Then for all $z\in \mathbb{P}\Delta(0;1_n)$, we have
\[|f(z)|\leq \frac{|f(0)|+||z||_{\infty}}{1+|f(0)|||z||_{\infty}}.\]
\end{lem}
The following lemma is contained in \cite[Corollary 1.3]{Liu-Chen-IJPAM-2012}.
\begin{lem}\label{Lem2} Let $f$ be holomorphic in the polydisk $\mathbb{P}\Delta(0;1_n)$ such that $|f(z)|< 1$ for all $z\in \mathbb{P}\Delta(0;1_n)$. Then for multi-index $\alpha=(\alpha_1,\ldots,\alpha_n)$, we have
\begin{align*}
\left|\frac{\partial^{|\alpha|} f(z)}{\partial z_1^{\alpha_1}\ldots \partial z_n^{\alpha_n}}\right|\leq \alpha!\frac{1-|f(z)|^2}{(1-||z||_{\infty}^2)^{|\alpha|}}(1+||z||_{\infty})^{|\alpha|-N}
\end{align*}
for all $z\in \mathbb{P}\Delta(0;1_n)$, where $N$ is the number of the indices $j$ such that $\alpha_j\neq 0$.
\end{lem}

The following lemma is contained in \cite[Lemma 2.1]{Liu-Chen-IJPAM-2012}.
\begin{lem}\label{Lem3} Let $f$ be holomorphic in the polydisk $\mathbb{P}\Delta(0;1_n)$ such that $|f(z)|\leq 1$ for all $z\in \mathbb{P}\Delta(0;1_n)$. Suppose $f(z)=a_0+\sum_{|\alpha|=1}^{\infty}a_{\alpha}z^{\alpha}$
for all $z\in \mathbb{P}\Delta(0;1_n)$. Then for any multi-index $\alpha$, we have 
\begin{align*}
|a_{\alpha}|\leq 1-|a_0|^2.
\end{align*}
\end{lem}

The following lemma is contained in \cite[Lemma 6.1.28]{Graham-Kohr}.
\begin{lem}\label{Lem4} Let $f$ be holomorphic in the polydisk $\mathbb{P}\Delta(0;1_n)$ such that $|f(z)|\leq 1$ for all $z\in \mathbb{P}\Delta(0;1_n)$. Suppose $k(\geq 1)$ is the zero-multiplicity of $f$ at $0$. Then
\begin{align*}
|f(z)|\leq ||z||_{\infty}^k\;\text{ for all}\; z\in \mathbb{P}\Delta(0;1_n).
\end{align*}
\end{lem}

The following lemma is contained in \cite[Corollary 1.3]{Liu-Chen-IJPAM-2012}.
\begin{lem}\label{Lem5} Let $f$ be holomorphic in the polydisk $\mathbb{P}\Delta(0;1_n)$ such that $|f(z)|< 1$ for all $z\in \mathbb{P}\Delta(0;1_n)$. Then for multi-index $\alpha=(\alpha_1,\ldots,\alpha_n)$, we have
\begin{align*}
\left|\frac{\partial^{|\alpha|} f(z)}{\partial z_1^{\alpha_1}\ldots \partial z_n^{\alpha_n}}\right|\leq \alpha!\frac{1-|f(z)|^2}{(1-||z||_{\infty}^2)^{|\alpha|}}(1+||z||_{\infty})^{|\alpha|-N}
\end{align*}
for all $z\in \mathbb{P}\Delta(0;1_n)$, where $N$ is the number of the indices $j$ such that $\alpha_j\neq 0$.
\end{lem}

\begin{lem}\label{Lem6}\cite[Lemma 2.3]{Hu-Wang-Long-2022}For $0 \le x \le x_0 \le 1$, it holds that
\[
\phi(x) := x + A(1 - x^2) \le \phi(x_0)
\quad \text{whenever } 0 \le A \le \frac{1}{2}, \tag{2.1}
\]
and similarly,
\[
\psi(x) := x^2 + A(1 - x^2) \le \psi(x_0)
\quad \text{whenever } 0 \le A \le 1. \tag{2.2}
\]
\end{lem}

\section{{\bf Proofs of the main results}}\label{Sec-4}
\begin{proof}[\bf Proof of Theorem \ref{Th-2.3}] 
By the given condition $f(z)=a_0+\sum_{|\alpha|=1}^{\infty} a_{\alpha} z^{\alpha}$ is holomorphic in the polydisk $\mathbb{P}\Delta(0;1_n)$ such that $|f(z)|\leq 1$ in $\mathbb{P}\Delta(0;1_n)$. 
Let $z\in\mathbb{P}\Delta(0;1/n)$ such that $|z_i|=r_i<1/n$ for $i=1,2,\ldots,n$. Let $\omega(z)=(\omega_1(z),\ldots,\omega_n(z))\in\mathcal{B}_{n,m}$. Then by Lemma \ref{Lem4}, we have 
\[|\omega_i(z)|\leq r_i^m\]
for $i=1,2,\ldots,n$. Clearly, 
\begin{align}\label{Eq-4.9}
	||\omega(z)||_{\infty}\leq {\bf r}^m,
\end{align}
where ${\bf r}=\max\{r_1,r_2,\ldots,r_n\}$. 
Consequently, by Lemma \ref{Lem1}, we obtain
\begin{align}
	\label{Eq-4.10} |f(\omega(z))| \leq \frac{||\omega(z)||_{\infty}+|a_0|}{1+|a_0|\;||\omega(z)||_{\infty}}\leq \frac{n{\bf r}^m+|a_0|}{1+|a_0|n{\bf r}^m}.
\end{align}

Also by Lemma \ref{Lem3}, we have
\bea\label{Eq-4.11}
 |a_{\alpha}| \leq 1-|a_0|^2,\eea
where $\alpha=(\alpha_1,\alpha_2,\ldots,\alpha_n)$ such that $|\alpha|=k$.
Now, using \eqref{Eq-4.10} and \eqref{Eq-4.11}, we obtain
\begin{align}
&\mathcal{A}(z,{\bf r}):=
t|f(\omega(z))|+(1-t)\sum\limits_{k=0}^{\infty}\sum\limits_{|\alpha|=k}|a_{\alpha}| |\omega(z)|^{\alpha} \nonumber\\
&\le t\frac{n{\bf r}^m+|a_0|}{1+n{\bf r}^m|a_0|}+(1-t)|a_0|+(1-t)(1-|a_0|^2)\sum\limits_{k=1}^{\infty} {\bf r}^{mk}\sum\limits_{|\alpha|=k}1\nonumber\\
&\le t\frac{n{\bf r}^m+|a_0|}{1+n{\bf r}^m|a_0|}+(1-t)|a_0|+(1-t)(1-|a_0|^2)\sum\limits_{k=1}^{\infty} ({n\bf r}^m)^{k}\nonumber\\
&\le t\frac{n{\bf r}^m+|a_0|}{1+n{\bf r}^m|a_0|}+(1-t)|a_0|+(1-t)(1-|a_0|^2)\frac{n{\bf r}^m}{1-n{\bf r}^m} \nonumber\\
&=\frac{t(n{\bf r}^m+|a_0|)(1-n{\bf r}^m)
+(1-t)(1+n{\bf r}^m|a_0|)\bigl[|a_0|(1-n{\bf r}^m)+(1-|a_0|^2)n{\bf r}^m\bigr]}
{(1+n{\bf r}^m|a_0|)(1-n{\bf r}^m)}\nonumber\\
&=1+\frac{\psi_{m,n,t}({\bf r})}{(1+n{\bf r}^m|a_0|)(1-n{\bf r}^m)}\nonumber\\
&\leq 1
\label{eq:3.2}
\end{align}
provided  $\psi_{m,n,t}({\bf r})\leq 0$, where
\begin{align*}
\psi_{m,n,t}({\bf r})=&(1-n{\bf r}^m)\bigl[t(n{\bf r}^m+|a_0|)-(1+n{\bf r}^m|a_0|)\bigr]\\&
+(1-t)(1+n{\bf r}^m|a_0|)\bigl[|a_0|(1-n{\bf r}^m)+(1-|a_0|^2)n{\bf r}^m\bigr].
\end{align*}

Set
\begin{align*}
\alpha(x)&=(1-n{\bf r}^m)\bigl[t(n{\bf r}^m+x)-(1+n{\bf r}^mx)\bigr]\\&
+(1-t)(1+n{\bf r}^mx)\bigl[x(1-n{\bf r}^m)+(1-x^2)n{\bf r}^m\bigr]\\
&=-(1-t)n^2{\bf r}^{2m}x^3-(1-t)n^2{\bf r}^{2m}x^2+\bigl[(1-n{\bf r}^m)^2+(1-t)n^2{\bf r}^{2m}\bigr]x\\
&\quad-n^2{\bf r}^{2m}t+2n{\bf r}^m-1,
\end{align*}
where $x=|a_0|\in[0,1)$.

To prove the inequality \eqref{eq:3.2}, it is sufficient to prove that
\begin{align*}
\alpha(x)\le 0 \quad \text{for}\; {\bf r}\le R_{m,n,t}.
\end{align*}

Note that
\begin{align*}
	\alpha'(x) = -3(1-t)n^2{\bf r}^{2m}x^2-2(1-t)n^2{\bf r}^{2m}x+(1-n{\bf r}^m)^2+(1-t)n^2{\bf r}^{2m}
\end{align*}
and
\begin{align*}
	\alpha''(x) =-2(1-t)n^2{\bf r}^{2m}(3x+1).
\end{align*}

Clearly $\alpha''(x) \leq 0$ for $x \in [0, 1]$ and so $\alpha'(x)$ is a decreasing function in $[0,1]$. Consequently
\begin{equation}
	\alpha'(x) \geq \alpha'(1)=(4t-3)n^2{\bf r}^{2m}-2n{\bf r}^m+1.
	\label{eq:3.4}
\end{equation}

We now consider following three cases.\vspace{1.2mm}

\noindent
{\bf Case 1.} Let $t \in [0, 3/4)$. By applying Descartes' Rule of Signs, we observe that the number of sign changes in the coefficients of the polynomial
	\begin{align*}Q_t(\rho) = (4t-3)\rho^{2} -2\rho +1, \quad \rho = n\mathbf{r}^m,
	\end{align*}
	is one.  We observe that 
	\begin{align*}
		Q_t(0) =1> 0\; \mbox{and}\;Q_t(1) =4(t-1)< 0.
	\end{align*} 

By the Intermediate Value Theorem, there exists exactly one root in the interval $(0, 1)$. Consequently, $Q_t(n\mathbf{r}^m) \geq 0$ for all $\mathbf{r} \in [0, R_{m,n,t}]$, where $R_{m,n,t}=\sqrt[m]{\frac{1-2\sqrt{1-t}}{n(4t-3)}}$ denotes the unique root in $(0,1)$ of the equation $Q_t(n\mathbf{r}^m) = 0$.\vspace{1.2mm}

Now inequality \eqref{eq:3.4} implies that $\alpha'(x) \geq 0$ for all $\mathbf{r} \in [0, R_{m,n,t}]$, indicating that $\alpha(x)$ is a non-decreasing function of $x$ on the interval $[0, 1]$. Consequently, we have $\alpha(x) \leq \alpha(1) = 0$ for all $\mathbf{r} \in [0, R_{m,n,t}]$ and $x \in [0, 1]$.\vspace{1.2mm}

\noindent
{\bf Case 2.} Let $t \in (3/4, 1)$. By applying Descartes' Rule of Signs, we observe that the number of sign changes in the coefficients of the polynomial
	\begin{align*}Q_t(\rho) = (4t-3)\rho^{2} -2\rho +1, \quad \rho = n\mathbf{r}^m,
	\end{align*}
	is two. Furthermore, since 
	\begin{align*}
		Q_t(0) =1> 0\; \mbox{and}\;Q_t(1) =4(t-1)< 0,
	\end{align*} the intermediate value theorem guarantees that the smallest positive root $R_{m,n,t}=\sqrt[m]{\frac{1-2\sqrt{1-t}}{n(4t-3)}}$ lies in the interval $(0, 1)$. It follows that $Q_t(n\mathbf{r}^m) \geq 0$ for all $\mathbf{r} \in [0, R_{m,n,t}]$.\vspace{1.2mm}

Therefore inequality \eqref{eq:3.4} implies that $\alpha'(x) \geq 0$ for all $\mathbf{r} \in [0, R_{m,n,t}]$ and $x \in [0, 1]$. Consequently, $\alpha(x)$ is a non-decreasing function of $x\in [0, 1]$. It follows that $\alpha(x) \leq \alpha(1) = 0$ for all $\mathbf{r} \in [0, R_{m,n,t}]$ and so our requirement is established.\vspace{1.2mm}

\noindent
{\bf Case 3.} Let $t=3/4$. Now from \eqref{eq:3.4}, we get $\alpha'(x) \geq -2n{\bf r}^m+1\geq 0$ for $\mathbf{r}\le \sqrt[m]{\frac{1}{2n}}$ and $x \in [0, 1]$. Consequently, $\alpha(x)$ is a non-decreasing function of $x\in [0, 1]$. It follows that $\alpha(x) \leq \alpha(1) = 0$ for all $\mathbf{r}\le \sqrt[m]{\frac{1}{2n}}$ and so our requirement is established.\vspace{1.2mm}

To show that the number $R_{m,n,t}$ is best possible, we let $a\in [0, 1)$ and consider the functions $\omega(z)=(z_1^m,\ldots,z_n^m)$
and 
\begin{align}\label{Eq-44.9}
	f_a(z)=\frac{a-(z_1+\ldots+z_n)}{1-a(z_1+\ldots+z_n)}.
\end{align}
It is easy to verify that $f$ is holomorphic in the polydisk $\mathbb{P}\Delta(0;1/n)$. Since 
\begin{align*}
	|a(z_1+z_2+\ldots+z_n)|<1,
\end{align*}
we see that $f_a$ has the following series expansion
\begin{align*}
	f_a(z)=a-(1-a^2)\sum\limits_{k=1}^{\infty} a^{k-1}(z_1+\cdots+z_n)^k
\end{align*}
for all $z\in\mathbb{P}\Delta(0;1/n)$. For the point 
\begin{align*}
	z=((-1)^{(2m-1)/m}r,(-1)^{(2m-1)/m}r,\ldots,(-1)^{(2m-1)/m}r),
\end{align*}
we find that
\begin{align}\label{BMN1}
	|f_a(\omega(z))|=\frac{a+n{\bf r}^m}{1+na{\bf r}^m}
\end{align}
and
\begin{align}\label{BMN2}
	f_a(\omega(z))=a-(1-a^2)\sum\limits_{k=1}^{\infty} a^{k-1}(-1)^{k(2m-1)}(n{\bf r}^m)^k.
\end{align}

Using (\ref{BMN1}) and (\ref{BMN2}), a simple computation shows that
\begin{align}
&t|f_a(\omega(z)|+(1-t)\sum_{k=0}^{\infty}\sum_{|\alpha|=k}|a_{\alpha}||\omega(z)|^{\alpha} \nonumber\\
&=t\frac{a+n{\bf r}^m}{1+an{\bf r}^m}+(1-t)a
+(1-t)(1-a^2)\frac{n{\bf r}^m}{1-an{\bf r}^m} \nonumber\\
&=\frac{t(a+n{\bf r}^m)(1-an{\bf r}^m)
+(1-t)(1+an{\bf r}^m)\bigl[a(1-an{\bf r}^m)+(1-a^2)n{\bf r}^m\bigr]}
{1-a^2n^2{\bf r}^{2m}}.
\label{eq:3.6}
\end{align}
To establish sharpness, it suffices to show that for any $\mathbf{r} > R_{m,n,t}$, there exists some $a \in [0, 1)$ such that the right-hand side of \eqref{eq:3.6} is strictly greater than $1$, which is equivalent to showing that
\begin{align*}
\mu(a,n,t,{\bf r})>0
\quad \text{for } {\bf r}>R_{m,n,t} \text{ and some } a\in[0,1),
\end{align*}
where
\begin{align*}
\mu(a,n,t,{\bf r})=&
(1-a)\Bigl[2n^2{\bf r}^{2m}(1-t)a^2-\left((2t-1)n^2{\bf r}^{2m}-n{\bf r}^m\right)a+n{\bf r}^m-1\Bigr].
\end{align*}

Let us define the auxiliary function $\rho(a)$ for $a \in [0, 1)$ as
	\begin{align*}
		\rho(a) :=2n^2{\bf r}^{2m}(1-t)a^2-\left((2t-1)n^2{\bf r}^{2m}-n{\bf r}^m\right)a+n{\bf r}^m-1.
	\end{align*}
	Differentiating with respect to $a$, we obtain
	\begin{align*}
		\rho'(a) &= 4n^2{\bf r}^{2m}(1-t)a-(2t-1)n^2{\bf r}^{2m}+n{\bf r}^m.
	\end{align*}
	
Observe that $4n^2{\bf r}^{2m}(1-t)a-(2t-1)n^2{\bf r}^{2m}+n{\bf r}^m$ is a continuous and increasing function of $a\in [0, 1)$ for each
fixed $t\in [0, 1]$ and $n{\bf r}\in [0, 1)$. Consequently $\rho'(a) \geq (1-2t)n^2{\bf r}^{2m}+n{\bf r}^m\geq 0$ for all $t \in [0, 1]$ and $n\mathbf{r} \in [0, 1)$, which implies that $\rho(a)$ is non-decreasing on the interval $[0, 1]$. Consequently, $\rho(a)$ attains its maximum at $a=1$:
	\begin{align*}
		\rho(a) \leq \rho(1) =-\left((4t-3)n^2{\bf r}^{2m}-2n{\bf r}^m+1\right),
	\end{align*}
	for each fixed $t \in [0, 1]$ and $n\mathbf{r} \in [0, 1)$.\vspace{1.2mm}

Furthermore, the monotonicity of $(4t-3)n^2{\bf r}^{2m}-2n{\bf r}^m+1$ leads that if $\mathbf{r} > R_{m,n,t}$, then $(4t-3)n^2{\bf r}^{2m}-2n{\bf r}^m+1<0$. Therefore if $\mathbf{r} > R_{m,n,t}$, then $\rho(1)>0$. By the continuity of $\rho(a)$ on the interval $[0,1]$, it follows that
\begin{align*}
	\lim_{a \to 1^-} \rho(a) = \rho(1) > 0.
\end{align*}
Consequently, for any fixed $\mathbf{r} > R_{m,n,t}$, there exists a sufficiently large $a \in [0, 1)$ such that $\rho(a) > 0$, which implies $\mu(a, n, t, \mathbf{r}) > 0$. This demonstrates that the right-hand side of \eqref{eq:3.6} is strictly greater than $1$ if $\mathbf{r} > R_{m,n,t}$, and thus the radius $R_{m,n,t}$ is sharp.
\end{proof}

\begin{proof}[\bf Proof of Theorem \ref{Th-2.4}]
By the given condition $f(z)=a_0+\sum_{|\alpha|=1}^{\infty} a_{\alpha} z^{\alpha}$ is holomorphic in the polydisk $\mathbb{P}\Delta(0;1_n)$ such that $|f(z)|\leq 1$ in $\mathbb{P}\Delta(0;1_n)$. 
Let $z\in\mathbb{P}\Delta(0;1/n)$ such that $|z_i|=r_i<1/n$ for $i=1,2,\ldots,n$. Let $\omega(z)=(\omega_1(z),\ldots,\omega_n(z))\in\mathcal{B}_{n,m}$.
According to Lemma~\ref{Lem5} and (\ref{Eq-4.9}), we have the estimate
\begin{align}\label{SB1}
	|\partial_u(f(\omega(z))| \leq \sum\limits_{j=1}^n |u_j|\frac{1-|f(\omega(z))|^2}{1-||\omega(z)||_{\infty}^2} \leq \frac{1-|f(\omega(z))|^2}{1-(n\mathbf{r}^m)^2}.
\end{align}

A direct calculation shows that
\begin{align*}
	\frac{n\mathbf{r}^m}{1-(n\mathbf{r}^m)^2} \leq \frac{1}{2} \quad \text{for } 0 \leq n\mathbf{r}^m \leq \sqrt{2}-1.
\end{align*}

Next, consider the auxiliary function
\begin{align*}
	\phi(x) =x +A(1-x^2),
\end{align*}
where $x = |f(\omega(z))|$ and $A= \frac{n\mathbf{r}^m}{1-(n\mathbf{r}^m)^2}$. Again from (\ref{Eq-4.10}), we get
\begin{align*}
x=|f(\omega(z))|\leq \frac{n\mathbf{r}^m + |a_0|}{1 + |a_0|n\mathbf{r}^m}=x_0.
\end{align*}

Since $A\leq 1/2$, using Lemma \ref{Lem6} and (\ref{SB1}), we deduce that
\begin{align}\label{SB0}
|f(\omega(z))|&+|\partial_u(f(\omega(z))|\;||\omega(z)||_{\infty}\\&\le |f(\omega(z))|+\frac{n\mathbf{r}^m}{1-(n\mathbf{r}^m)^2}\left(1-|f(\omega(z))|^2\right)\nonumber\\&=
\phi(x)\nonumber\\&\le \phi(x_0)\nonumber\\&=
\frac{n{\bf r}^m +|a_{0}|}{1+n{\bf r}^m|a_{0}|} + \frac{n{\bf r}^m}{1-(n{\bf r}^m)^{2}}
\left( 1 - \left( \frac{n{\bf r}^m+|a_{0}|}{1+n{\bf r}^m|a_{0}|}\right)^{2} \right).\nonumber
\end{align}

Furthermore, applying Lemma~\ref{Lem3} yields
\begin{align}\label{SBB1}
	\sum_{k=2}^{\infty} \sum_{|\alpha|=k} |a_{\alpha}||\omega(z)|^{\alpha} \leq (1-|a_0|^2) \sum_{k=2}^{\infty} (n\mathbf{r}^m)^k = (1-|a_0|^2) \frac{(n\mathbf{r}^m)^2}{1-n\mathbf{r}^m}.
\end{align}

Substituting the estimates \eqref{SB0}, and \eqref{SBB1} into the main inequality, we obtain
\begin{align}
\mathcal{B}(z,{\bf r})\nonumber&
\le  
\frac{n{\bf r}^m +|a_{0}|}{1+n{\bf r}^m|a_{0}|} + \frac{n{\bf r}^m}{1-(n{\bf r}^m)^{2}}
\left( 1 - \left( \frac{n{\bf r}^m+|a_{0}|}{1+n{\bf r}^m|a_{0}|}\right)^{2} \right)
+ \lambda\frac{(1-|a_{0}|^{2})(n{\bf r}^m)^{2}}{1-n{\bf r}^m}\
\nonumber\\&=
\frac{n{\bf r}^m +|a_{0}|}{1+n{\bf r}^m|a_{0}|}
+ \frac{n{\bf r}^m(1-|a_{0}|^{2})}{(1+n{\bf r}^m|a_{0}|)^{2}}
+ \lambda\frac{(1-|a_{0}|^{2})(n{\bf r}^m)^{2}}{1-n{\bf r}^m}\nonumber\\&
= 1 + \frac{\delta(n{\bf r}^m)}{(1+|a_{0}|n{\bf r}^m)^{2}(1-n{\bf r}^m)},\label{eq:4.1}
\end{align}
for ${\bf r}\le \sqrt[m]{\frac{\sqrt{2}-1}{n}}$, where 
\begin{align*}
&\delta(n{\bf r}^m)\\&
=(1-|a_0|)
\Bigl[(1-n{\bf r}^m)(|a_0|(n{\bf r}^m)^2+2n{\bf r}^m-1)
+\lambda (n{\bf r}^m)^2(1+|a_0|)(1+|a_0|n{\bf r}^m)^2\Bigr].
\end{align*}
\vspace{1.2mm}

To prove the inequality \eqref{eq:4.1}, it is sufficient to prove that $\delta(n{\bf r}^m)\le 0 \quad \text{for}\; {\bf r}\le R_{m,n,\lambda}.$ Observe that
\begin{align*}
\delta(n{\bf r}^m)
&\le (1-|a_0|)
\Bigl[(1-n{\bf r}^m)((n{\bf r}^m)^2+2n{\bf r}^m-1)+2\lambda (n{\bf r}^m)^2(1+n{\bf r}^m)^2\Bigr]
\\&=(1-|a_0|)\,\xi(n{\bf r}^m),
\end{align*}
where
\begin{align*}
\xi(n{\bf r}^m)=2\lambda (n{\bf r}^m)^4+(4\lambda-1)(n{\bf r}^m)^3+(2\lambda-1)(n{\bf r}^m)^2+3n{\bf r}^m-1.
\end{align*}
It is enough to prove that $\xi(n{\bf r}^m)\le 0$ for ${\bf r}\le R_{m,n,\lambda}$.\vspace{1.2mm}

 Next, we divide it into two cases to discuss.\vspace{1.2mm}

\noindent
\textbf{Case 1.} Let $\lambda \in (1/2, \infty)$. We define the auxiliary function $w(n\mathbf{r}^m)$ such that
\begin{align*}
	\xi(n\mathbf{r}^m) > (n\mathbf{r}^m)^4 + (n\mathbf{r}^m)^3 + 3n\mathbf{r}^m - 1 =: w(n\mathbf{r}^m), \quad n\mathbf{r}^m \in [0, 1).
\end{align*}
A direct calculation shows that $\xi(\sqrt{2}-1) > w(\sqrt{2}-1) = 6 - 4\sqrt{2} > 0$. Since $\xi(n\mathbf{r}^m)$ is a monotonically increasing function of $n\mathbf{r}^m$ on $[0, 1)$ and satisfies $\xi(0) = -1 < 0$, the intermediate value theorem guarantees the existence of a unique $\mathbf{r}_\lambda \in \left(0, \sqrt[m]{\frac{\sqrt{2}-1}{n}}\right)$ such that $\xi(n\mathbf{r}_\lambda^m) = 0$. Consequently, we have $\xi(n\mathbf{r}^m) \leq 0$ for all $n\mathbf{r}^m \in [0, n\mathbf{r}_\lambda^m]$.

\medskip
\noindent
\textbf{Case 2.} Let $\lambda \in (0, 1/2]$. In this case, the auxiliary function $\xi(n\mathbf{r}^m)$ satisfies the inequality
\begin{align*}
	\xi(n\mathbf{r}^m) \leq w(n\mathbf{r}^m) := (n\mathbf{r}^m)^4 + (n\mathbf{r}^m)^3 + 3n\mathbf{r}^m - 1, \quad n\mathbf{r}^m \in [0, 1).
\end{align*}
Note that $w(n\mathbf{r}^m)$ is strictly increasing on the interval $[0, 1)$, with $w(0) = -1 < 0$ and $w(\sqrt{2}-1) = 6 - 4\sqrt{2} > 0$. By the intermediate value theorem, there exists a unique root $\mathbf{r}_* \in \left(0, \sqrt[m]{\frac{\sqrt{2}-1}{n}}\right)$ such that $w(n\mathbf{r}_*^m) = 0$. Consequently, $w(n\mathbf{r}^m) \leq 0$ for all $n\mathbf{r}^m \in [0, n\mathbf{r}_*^m]$, which implies that $\xi(n\mathbf{r}^m) \leq 0$ on the same interval.

\medskip
To show that the number $R_{m,n,\lambda}$ is best possible, we let $a\in [0,1)$ and consider the
functions $\omega(z)=(z_1^m,z_2^m\ldots,z_n^m)$ and $f_a$ as given in \eqref{Eq-44.9}. For the point 
\begin{align*}
	z=((-1)^{(2m-1)/m}r,(-1)^{(2m-1)/m}r,\ldots,(-1)^{(2m-1)/m}r),
\end{align*}
 we find from (\ref{BMN1}) and (\ref{BMN2}) that
\begin{align}
|f_a(\omega(z))| &+|\partial_u(f_a(\omega(z)))|\;||\omega(z)||_{\infty} 
+ \lambda \sum_{k=2}^{\infty} \sum_{|\alpha|=k}^{\infty} |a_{\alpha}|\, |\omega(z)|^{\alpha}\nonumber\\=&
\frac{n{\bf r}^m+a}{1+n{\bf r}^ma}
+ \frac{(1-a^{2})n{\bf r}^m}{(1+an{\bf r}^m)^{2}}+\lambda\frac{(1-a^{2})a(n{\bf r}^m)^{2}}{1-an{\bf r}^m}
\nonumber \\
=& \frac{
\bigl(a + n{\bf r}^m\bigr)\bigl(1 - a^{2}(n{\bf r}^m)^{2}\bigr)
+ (1 - a^{2})\Bigl[n{\bf r}^m(1 - a n{\bf r}^m)
+ \lambda a (n{\bf r}^m)^{2}(1 + a n{\bf r}^m)^{2}\Bigr]
}{
(1 + a n{\bf r}^m)(1 - a^{2} (n{\bf r}^m)^{2})
}.\label{eq:3.14}
\end{align}
To establish the sharpness of the result, it remains to show that for any $\mathbf{r} > R_{m,n,\lambda}$, there exists a parameter $a \in [0, 1)$ such that the right-hand side of \eqref{eq:3.14} exceeds unity. This is equivalent to demonstrating that
\begin{align}\label{eq:3.15}
&(1-a)\Big[\lambda (n{\bf r}^m)^4a^4+(\lambda (n{\bf r}^m)^4+2\lambda (n{\bf r}^m)^3)a^3
+\bigl((2\lambda-1)(n{\bf r}^m)^3+\lambda (n{\bf r}^m)^2\bigr)a^2\nonumber \\
&\quad+\bigl((\lambda-1)(n{\bf r}^m)^2+n{\bf r}^m\bigr)a+2n{\bf r}^m-1\Big]>0
\end{align}
 ${\bf r}>R_{m,n,\lambda}$.\vspace{1.2mm} 
 
For $a \in [0, 1)$, we define the auxiliary function $K(a, \lambda, n\mathbf{r}^m)$ as follows:
 	\begin{align*}
 		K(a, \lambda, n\mathbf{r}^m) &:= \lambda (n\mathbf{r}^m)^4 a^4 + \left[ \lambda (n\mathbf{r}^m)^4 + 2\lambda (n\mathbf{r}^m)^3 \right] a^3 + \left[ (2\lambda - 1)(n\mathbf{r}^m)^3 + \lambda (n\mathbf{r}^m)^2 \right] a^2 \\
 		&\quad + \left[ (\lambda - 1)(n\mathbf{r}^m)^2 + n\mathbf{r}^m \right] a + 2n\mathbf{r}^m - 1.
 \end{align*}
 To establish the sharpness of the radius $R_{n,\lambda}$, it remains to demonstrate that for $n\mathbf{r} > R_{n,\lambda}$, there exists an $a \in [0, 1)$ such that $K(a, \lambda, n\mathbf{r}) > 0$. To this end, we consider the following cases.

\medskip
\noindent
\textbf{Case 1.}  Let $\lambda \in (1/2, \infty)$. For each fixed $n\mathbf{r}^m \in [0, 1)$, it is evident that $K(a, \lambda, n\mathbf{r}^m)$ is a monotonically increasing function of $a \in [0, 1)$. Thus, for $a \in [0, 1)$, we have the estimate
\begin{align*}
	K(a, \lambda, n\mathbf{r}^m)& \leq K(1, \lambda, n\mathbf{r}^m)\\ &= 2\lambda (n\mathbf{r}^m)^4 + (4\lambda - 1)(n\mathbf{r}^m)^3 + (2\lambda - 1)(n\mathbf{r}^m)^2 + 3n\mathbf{r}^m - 1= \xi(n\mathbf{r}^m).
\end{align*}
Recall that $\xi(n\mathbf{r}^m)$ is monotonically increasing on $[0, 1)$ and $\xi(n\mathbf{r}_\lambda^m) = 0$. Consequently, if $n\mathbf{r}^m > n\mathbf{r}_\lambda^m$, it follows that $\xi(n\mathbf{r}^m)>0$. By the continuity of $K$ with respect to $a$, we observe that
\begin{align*}
	\lim_{a \to 1^-} K(a, \lambda, n\mathbf{r}^m) = K(1, \lambda, n\mathbf{r}^m) = \xi(n\mathbf{r}^m) > 0
\end{align*}
for $n\mathbf{r}^m > n\mathbf{r}_\lambda^m$. This ensures that for any $n\mathbf{r}^m > n\mathbf{r}_\lambda^m$, there exists a sufficiently large $a \in [0, 1)$ such that $K(a, \lambda, n\mathbf{r}^m)>0$, which implies that the inequality \eqref{eq:3.15} is satisfied.

\medskip
\noindent
\textbf{Case 2.} Let $\lambda \in (0, 1/2]$. In this case, for $a \in [0, 1)$, the auxiliary function $K$ satisfies the following inequality:
\begin{align*}
	K(a, \lambda, n\mathbf{r}^m) \leq K(a, 1/2, n\mathbf{r}^m) \leq (n\mathbf{r}^m)^4 + (n\mathbf{r}^m)^3 + 3n\mathbf{r}^m - 1 = w(n\mathbf{r}^m).
\end{align*}
The strict monotonicity of $w(n\mathbf{r}^m)$ on $[0, 1)$ implies that $w(n\mathbf{r}^m)>0$ whenever $n\mathbf{r}^m>n\mathbf{r}_*^m$, where $n\mathbf{r}_*^m$ is the unique positive root of $w(n\mathbf{r}^m)=0$. It follows from the definition of $K$ that
\begin{align*}
	\lim_{a \to 1^-} K(a, \lambda, n\mathbf{r}^m) = w(n\mathbf{r}^m) > 0
\end{align*}
for $n\mathbf{r}^m > n\mathbf{r}_*^m$. By the continuity of $K$ with respect to $a$ on $[0, 1]$, for any $n\mathbf{r}^m> n\mathbf{r}_*^m$, we can find a sufficiently large $a \in [0, 1)$ such that $K(a, \lambda, n\mathbf{r}^m) > 0$. This confirms that the inequality \eqref{eq:3.15} is satisfied, thereby establishing the sharpness of the radius $\mathbf{r}_*^m$.
\end{proof}

\begin{proof}[\bf Proof of Theorem \ref{Th-2.5}]
By the given condition $f(z)=a_0+\sum_{|\alpha|=1}^{\infty} a_{\alpha} z^{\alpha}$ is holomorphic in the polydisk $\mathbb{P}\Delta(0;1_n)$ such that $|f(z)|\leq 1$ in $\mathbb{P}\Delta(0;1_n)$. 
Let $z\in\mathbb{P}\Delta(0;1/n)$ such that $|z_i|=r_i<1/n$ for $i=1,2,\ldots,n$. Let $\omega(z)=(\omega_1(z),\ldots,\omega_n(z))\in\mathcal{B}_{n,m}$.

A direct calculation shows that
\begin{align*}
	\frac{n\mathbf{r}^m}{1-(n\mathbf{r}^m)^2} \leq 1 \quad \text{for } 0 \leq n\mathbf{r}^m \leq \frac{\sqrt{5}-1}{2}.
\end{align*}

Next, consider the auxiliary function
\begin{align*}
	\phi(x) =x^2 +A(1-x^2),
\end{align*}
where $x = |f(\omega(z))|$ and $A= \frac{n\mathbf{r}^m}{1-(n\mathbf{r}^m)^2}$. Again from (\ref{Eq-4.10}), we get
\begin{align*}
x=|f(\omega(z))|\leq \frac{n\mathbf{r}^m + |a_0|}{1 + |a_0|n\mathbf{r}^m}=x_0.
\end{align*}

Since $A\leq 1/2$, using Lemma \ref{Lem6} and (\ref{SB1}), we deduce that
\begin{align}\label{SBB0}
|f(\omega(z))|^2&+|\partial_u(f(\omega(z))|\;||\omega(z)||_{\infty}\\&\le |f(\omega(z))|^2+\frac{n\mathbf{r}^m}{1-(n\mathbf{r}^m)^2}\left(1-|f(\omega(z))|^2\right)\nonumber\\&=
\phi(x)\nonumber\\&\le \phi(x_0)\nonumber\\&=
\left(\frac{n{\bf r}^m +|a_{0}|}{1+n{\bf r}^m|a_{0}|}\right)^2 + \frac{n{\bf r}^m}{1-(n{\bf r}^m)^{2}}
\left( 1 - \left( \frac{n{\bf r}^m+|a_{0}|}{1+n{\bf r}^m|a_{0}|}\right)^{2} \right).\nonumber
\end{align}

Substituting the estimates \eqref{SBB1} and \eqref{SBB0}  into the main inequality, we obtain
\begin{align}\label{DP1}
\mathcal{C}(z,{\bf r})\nonumber&
\le  
\left(\frac{n{\bf r}^m +|a_{0}|}{1+n{\bf r}^m|a_{0}|}\right)^2 + \frac{n{\bf r}^m}{1-(n{\bf r}^m)^{2}}
\left( 1 - \left( \frac{n{\bf r}^m+|a_{0}|}{1+n{\bf r}^m|a_{0}|}\right)^{2} \right)
+ \lambda\frac{(1-|a_{0}|^{2})(n{\bf r}^m)^{2}}{1-n{\bf r}^m}\
\nonumber\\&=
\left(\frac{n{\bf r}^m +|a_{0}|}{1+n{\bf r}^m|a_{0}|}\right)^2
+ \frac{n{\bf r}^m(1-|a_{0}|^{2})}{(1+n{\bf r}^m|a_{0}|)^{2}}
+ \lambda\frac{(1-|a_{0}|^{2})(n{\bf r}^m)^{2}}{1-n{\bf r}^m}\nonumber\\&
= 1 + \frac{(1-|a_0|)\Psi(|a_0|, n{\bf r}^m,\lambda)}{(1+|a_{0}|n{\bf r}^m)^{2}(1-n{\bf r}^m)}
\end{align}
for ${\bf r}\le \sqrt[m]{\frac{\sqrt{5}-1}{2n}}$, where 
\begin{align*}
\Psi(|a_0|, n{\bf r}^m,\lambda)= n^4{\bf r}^{4m}\lambda |a_0|^{2} + 2n^3{\bf r}^{3m}\lambda |a_0| - n^3{\bf r}^{3m} + n^2{\bf r}^{2m}\lambda + 2n{\bf r}^m - 1.
\end{align*}

Observe that $\Psi(|a_0|, n{\bf r}^m,\lambda)$ is a monotonically increasing function of  $|a_0| \in [0,1)$ for each fixed $\lambda \in [0,\infty)$ and $n{\bf r} \in (0,1)$. Then we have

\begin{align*}
 \Psi(|a_0|, n{\bf r}^m,\lambda) \le \Psi(1, n{\bf r}^m,\lambda)
=n^4{\bf r}^{4m}\lambda+ 2n^3{\bf r}^{3m}\lambda- n^3{\bf r}^{3m} + n^2{\bf r}^{2m}\lambda + 2n{\bf r}^m - 1.
\end{align*}

It is enough to prove that $\Phi(1,n{\bf r}^m,\lambda) \le 0$ holds for  ${\bf r} \le \tilde R_{m,n,\lambda}$.
Next, we divide it into two cases to discuss.\vspace{1.2mm}

\noindent
\textbf{Case 1.} Let $\lambda \in (1,\infty)$. Clearly $\Psi(1,n{\bf r}^m, \lambda)$ is a continuous and increasing function of $n{\bf r}^m \in (0,1)$. Note that
\begin{align*}
\Psi(1,0,\lambda)= -1< 0\quad \text{and}\quad \Psi\left(1,\sqrt[m]{\frac{\sqrt{5}-1}{2n}},\lambda\right)=\lambda > 0.
\end{align*}

Intermediate value theorem guarantees the existence of a unique $\tilde {\bf r}_{\lambda}\in \left(0,\sqrt[m]{\frac{\sqrt{5}-1}{2n}}\right)$ such that $\Psi(1,n\tilde {\bf r}_{\lambda}^m,\lambda)=0$. Consequently, we have $\Psi(1,n{\bf r}^m,\lambda)\le 0$ for all $n{\bf r}^m\in [0,n\tilde {\bf r}_{\lambda}^m]$. 

\medskip
\noindent
\textbf{Case 2.} Let $\lambda \in [0,1]$. Note that
\begin{align*}
\Psi(1,n{\bf r}^m,\lambda) \le n^4{\bf r}^{4m} + n^3{\bf r}^{3m} + n^2{\bf r}^{2m} + 2n{\bf r}^m - 1.
\end{align*}

Let $S(n{\bf r}^m)= n^4{\bf r}^{4m} + n^3{\bf r}^{3m} + n^2{\bf r}^{2m} + 2n{\bf r}^m - 1.$ Clearly, $S(n{\bf r}^m)$ is a continuous and increasing function of $n{\bf r}^m \in (0,1)$. Note that
\begin{align*}
S(0) = -1<0\quad \text{and}\quad S\left(\sqrt[m]{\frac{\sqrt{5}-1}{2n}}\right) = 1 > 0.
\end{align*}
Intermediate value theorem guarantees the existence of a unique $\tilde {\bf r}_{*}\in \left(0,\sqrt[m]{\frac{\sqrt{5}-1}{2n}}\right)$ such that $S(n\tilde {\bf r}_{*}^m)=0$. Consequently, we have $S(n{\bf r}^m)\le 0$ for all $n{\bf r}^m\in [0,n\tilde {\bf r}_{*}^m]$ and so
\begin{align*}
\Psi(1,n{\bf r}^m, \lambda) \le S(n{\bf r}^m) \le 0 \quad \text{for } n{\bf r}^m\in [0,n\tilde {\bf r}_{*}^m].
\end{align*}
Hence, inequality (\ref{DP1}) holds for ${\bf r}\le \tilde R_{m,n,\lambda}$.\vspace{1.2mm}

The sharpness part is similar to Theorem \ref{Th-2.5} and so we omit it.
\end{proof}

\noindent\textbf{Conflict of interest:} The authors declare that there is no conflict  of interest regarding the publication of this paper.\vspace{1.2mm}

\noindent {\bf Funding:} Not Applicable.\vspace{1.2mm}

\noindent\textbf{Data availability statement:}  Data sharing not applicable to this article as no datasets were generated or analysed during the current study.\vspace{1.2mm}

\noindent {\bf Authors' contributions:} All the authors have equal contributions in preparation of the manuscript.

\end{document}